\documentclass[letterpaper, 10 pt, conference]{ieeeconf}  

\IEEEoverridecommandlockouts                              
\usepackage{bbm}
\usepackage{mathrsfs}
\usepackage{verbatim}
\usepackage{amsmath}
\usepackage{amsfonts}
\usepackage{bm}
\usepackage{graphicx}
\usepackage{float}

\usepackage{amsthm}
\usepackage{xcolor}
\usepackage[colorinlistoftodos]{todonotes}
\usepackage{mathtools}
\usepackage{cite}
\usepackage{colortbl}
\usepackage{subcaption}
\theoremstyle{definition}
\newtheorem{assumption}{Assumption}
\newtheorem{definition}{Definition}
\newtheorem{remark}{Remark}
\newtheorem{problem}{Problem}

\theoremstyle{plain}

\newtheorem{theorem}{Theorem}
\newtheorem{lemma}{Lemma}

\DeclareMathOperator*{\argmin}{arg\, min}

\title{\LARGE \bf Stochastic Control with Distributionally Robust Constraints for Cyber-Physical Systems Vulnerable to Attacks}

\author{Nishanth Venkatesh$^{1}$, {\itshape{Student Member, IEEE}}, Aditya Dave$^{2}$, {\itshape{Member, IEEE}}, \\
Ioannis Faros$^{1}$, {\itshape{Student Member, IEEE}},
Andreas A. Malikopoulos$^{1,2}$, {\itshape{Senior Member, IEEE}}
	\thanks{This research was supported by NSF under Grants CNS-2149520 and CMMI-2219761.}
    \thanks{$^{1}$Department of Systems Engineering, Cornell University, Ithaca, NY 14850 USA.}
    \thanks{$^{2}$School of Civil and Environmental Engineering, Cornell University, Ithaca, NY 14853 USA (email: \texttt{ns942@cornell.edu; a.dave@cornell.edu; if74@cornell.edu; amaliko@cornell.edu).}} }

\begin{document}

\maketitle
\thispagestyle{empty}

\begin{abstract}

In this paper, we investigate the control of a cyber-physical system (CPS) while accounting for its vulnerability to external attacks.
We formulate a constrained stochastic problem with a robust constraint to ensure robust operation against potential attacks.
We seek to minimize the expected cost subject to a constraint limiting the worst-case expected damage an attacker can impose on the CPS.
We present a dynamic programming decomposition to compute the optimal control strategy in this robust-constrained formulation and prove its recursive feasibility.
We also illustrate the utility of our results by applying them to a numerical simulation.
\end{abstract}

\section{Introduction}
\label{section:Introduction}

Cyber-physical systems (CPSs) have enabled highly efficient control of physical processes by tightly coupling sensing, communication, and computational processing to generate real-time decisions with classical \cite{kim2012cyber} and nonclassical information structures \cite{Malikopoulos2021}.
They span various important applications including, but not limited to, connected and automated vehicles \cite{Malikopoulos2020,Nishanth2023AISmerging}, 
Internet of Things \cite{ansere2020optimal}, and 
social media platforms \cite{Dave2020SocialMedia}. 
However, in each of these applications, the interplay between the cyber components and the physical world can make the system vulnerable to various security threats, e.g., control system malware \cite{baezner2017stuxnet} and staged attacks \cite{serror2020challenges}. 
This has led to many studies on controlling CPSs while ensuring robustness and resilience to attacks \cite{ghiasi2023comprehensive,10155711}. 

The common modeling framework for CPSs utilizes a stochastic formulation to account for uncertainties in the dynamics that arise within the evolution of the physical process.
In this formulation, an agent is assumed to have access to a prior distribution for all uncertainties and must compute a control strategy to generate real-time control actions that minimize the total expected cost
\cite{Dave2021nestedaccess, sutton2018reinforcement}.
In stochastic formulations, constraints on the state and actions are modeled as probabilistic constraints, which can be imposed either in expectation or with some probabilities \cite{varagapriya2023chance}.
Similarly, approaches like those reported in  \cite{altman2021constrained}, \cite{ermon2012probabilistic} consider probabilistic constraints on the cumulative reward.
However, the actual performance and constraint satisfaction of an optimal control strategy are very sensitive to changes in a mismatch between the assumed prior probability model and the actual model \cite{Malikopoulos2022a,mannor2007bias}. 
Such a mismatch is bound to occur when a CPS is under attack from an adversary. Thus, it may not be appropriate to model safety-critical requirements on system behavior using probabilistic constraints in a stochastic formulation.

To accommodate the needs of safety-critical systems, several research efforts \cite{bertsekas1973sufficiently,iyengar2005robust,gagrani2017decentralized,shoukry2013minimax} have explored minimax formulations. Similar approaches \cite{Dave2021minimax, Dave2023approximate} consider non-stochastic formulations in which the agent does not have knowledge about the distributions of uncertainties and uses only the set of feasible values to compute optimal strategies that minimize the maximum costs. 
Though such approaches are suitable for applications under attack, such as cyber-security \cite{rasouli2018scalable}, and power systems \cite{zhu2011robust}, during regular operation of systems without attacks, they lead to outcomes that are overly conservative \cite{coraluppi2000mixed}.
Consequently, there remains a need for alternative approaches towards the control of vulnerable systems that avoid overly conservative decision-making during regular system operation and maintain a level of reliability when the system is occasionally attacked.

In this paper, we combine the superior performance of stochastic formulations in achieving an objective and the safety guarantees of worst-case formulations in minimizing vulnerabilities. 
To this end, we impose a distributionally robust constraint on a secondary objective that accounts for the vulnerabilities of a CPS to an attack. 
Concurrently, we aim to minimize the expected value of a primary cost for the best performance over a finite horizon.
Our formulation generalizes the  previous work of \cite{chen2004dynamic}, which addressed the problem of minimizing an expected discounted cost subject to either an expected or a minimax constraint. 
By considering a distributionally robust constraint, our formulation allows for greater control over the trade-off between conservativeness and optimality by appropriately adjusting the size of the uncertainty set for probability distributions.
In the extreme case that the set of feasible distributions is a singleton, we recover an expected value constraint. In contrast, if we expand the set to allow every possible distribution on the state space, we recover the non-stochastic worst-case constraint as a special case.
Thus, by changing the set of feasible distributions, we can better select the level of conservativeness of our formulation.

Our main contributions in this paper are (1) the problem formulation of controlling a vulnerable CPS using a stochastic cost and distributionally robust constraint (Problem \ref{problem_1}), (2)  a dynamic programming (DP) decomposition for this problem, which computes the optimal strategy that ensures recursive feasibility of the constraint (Theorem \ref{theorem_dp}), and (3) the illustration of the utility of our results by comparing them to both stochastic and worst-case approaches in numerical simulation (Section \ref{section:numerical_example}).

The remainder of the paper proceeds as follows. In Section \ref{section:problem}, we formulate the problem. In Section \ref{section:feasible_sets_DP}, we present the DP decomposition. In Section \ref{section:numerical_example}, we demonstrate our results in a numerical example, and in Section \ref{section:conclusion}, we draw concluding remarks. 

\section{Model}
\label{section:problem}

We consider a CPS whose evolution is described by a finite Markov decision process (MDP), denoted by a tuple $(\mathcal{X},\mathcal{U}, n, P, c, c_n)$, where $\mathcal{X}$ is a finite state space and $\mathcal{U}$ is a finite set of feasible actions available to an agent seeking to control the MDP. The system evolves over discrete time steps denoted by $t=0, \dots, n$, where $n \in \mathbb{N}$ is the finite time horizon. 
The state of the system and the control action of the agent at each $t$ are denoted by the random variables $X_t$ and $U_t$, respectively.
The transition function  at each $t$ is denoted by
$P_t: \mathcal{X} \times \mathcal{U} \times \mathcal{X} \to \Delta(\mathcal{X})$, where $\Delta(\mathcal{X})$ is the set of all probability distributions on the state space $\mathcal{X }$.
Nominally, the transition function  is given by $P_t = \Bar{P}$ for all $t$, where $\Bar{P} \in \Delta(\mathcal{X})$.
For the realizations $x_t \in \mathcal{X}$ and $u_t \in \mathcal{U}$ of the state $X_t$ and the control action $U_t$, the probability of transitioning to a state $x_{t+1} \in \mathcal{X}$ is $\mathbb{P}(X_{t+1}=x_{t+1} \,|\, x_t,u_t) = P_t(x_{t+1}\,|\,x_t,u_t)$. 
The agent selects the action using a control law $g_t: \mathcal{X} \to \mathcal{U}$ as $U_t = g_t(X_t)$, where $g_t$ is chosen from the feasible set of control laws at time $t$, denoted as $\mathcal{G}_t$.
The tuple of control laws denotes the control strategy of the agent $\boldsymbol{g} := (g_0,\dots,g_{n-1})$, where $\boldsymbol{g} \in \mathcal{G}$ and $\mathcal{G} = \prod_{t = 0}^{n-1}\mathcal{G}_{t}$.
After selecting the action at each $t=0,\dots,n-1$, the agent incurs a cost $c(X_t, U_t)$ generated using the function $c: \mathcal{X} \times \mathcal{U} \to \mathbb{R}$.  
Then, the performance of a strategy $\boldsymbol{g}$ is measured by the total expected cost beginning at an initial state $x_0 \in \mathcal{X}$:
\begin{equation}
    \label{performance_g}
    \mathcal{J}_0({\boldsymbol{g}}; x_0) = \mathbb{E}^{\boldsymbol{g}}\Bigg[\sum_{t=0}^{n-1} c(X_t,U_t) + c_n(X_n) \,\Big|\,x_0\Bigg],
\end{equation}
where $c_n: \mathcal{X} \to \mathbb{R}$ is the terminal cost, and $\mathbb{E}^{\boldsymbol{g}}$ denotes the expectation on all the random variables with respect to the probability distributions generated by the choice of control strategy $\boldsymbol{g}$. 



In the context of a CPS, the conventional approach of selecting a control strategy $\boldsymbol{g}$ to minimize the total expected cost \eqref{performance_g} may not be adequate to ensure smooth operation, particularly when the CPS is vulnerable to attacks by an adversary.
We consider that the presence or absence of an adversary during the system's operation is determined at the onset; however, this information is unknown to the agent.
The adversary's influence on the system's dynamics results in a change in the transition probability at each $t=0,\dots,n-1$ from a known set $\mathcal{P} \subseteq \Delta(\mathcal{X})$. Thus, an attack may be reflected by the choice of the worst transition function from $\mathcal{P}$. We allow the adversary to attack the system with access to the realization of the state $x_t \in \mathcal{X}$ and action $u_t \in \mathcal{U}$.
Note that the nominal transition function  $\Bar{P}$ belongs to the set $\mathcal{P}$ to allow for the case of no attack.

An agent that observes the presence or absence of an adversary can select either a purely robust or risk-neutral formulation, depending on the current situation.
However, a risk-neutral formulation may involve an arbitrarily large risk for the agent and leave the CPS vulnerable during an attack. In contrast, a robust formulation may be too conservative for the majority of situations where no attack occurs.
Thus, we impose a robust constraint to \textit{limit} the worst-case damage possible during an attack while minimizing the expected total cost.
To this end, the agent incurs a constraint penalty $d(X_t, U_t) \in \mathbb{R}$ at each $t=0,\dots,n-1$. The total expected worst-case penalty is given by
\begin{multline}
    \label{exp_max_constraint}
    \mathcal{L}_0(\boldsymbol{g};x_0) = \\ \max_{P_{0:n-1}\in \mathcal{P}^n} \mathbb{E}^{\boldsymbol{g}}_{P_{0:n-1}}\Bigg[\sum_{t=0}^{n-1} d(X_t,U_t)
    + d_n(X_n)\,\Big|\,x_0\Bigg],
\end{multline}
where $d_n:\mathcal{X}\rightarrow \mathbb{R}$ is the terminal penalty, $P_{0:n-1}$ is the collection of transition functions for $t=0, \ldots, n-1$, each taking values in the set $\mathcal{P}$. 
Note that this penalty has a distributionally robust form where the attacker may select the worst transition function $P_t \in \mathcal{P}$ at each $t$.
Furthermore, the choice of a particular function at any time $t$ does not limit the functions available to the adversary at time $t+1$ in \eqref{exp_max_constraint}.
The distributionally robust constraint is formulated by defining an upper bound $l_0 \in \mathbb{R}$, on the worst-case total expected penalty.  

\begin{remark}
During an attack, the agent may prioritize a different property, e.g., safety, of the system rather than the total expected cost used in \eqref{performance_g}. Hence, the constraint penalty at each instance of time is considered to be distinct from the cost. However, if we seek to limit the influence of the adversary on the performance itself, the penalty in the constraint can be set equal to the cost at each $t$. 
\end{remark}

Next, we define the agent's constrained control problem.

\begin{problem} \label{problem_1}
The optimization problem is to compute the optimal control strategy $\boldsymbol{g}^* \in \mathcal{G}$, if one exists, subject to a constraint on \eqref{exp_max_constraint}, i.e.,
\begin{align}
\label{optimal_g}
\min_{\boldsymbol{g}\in \mathcal{G}} \, &\mathcal{J}_0(\boldsymbol{g};x_0),\\
\text{s.t.} \quad &\mathcal{L}_0(\boldsymbol{g};x_0)\leq l_0,\label{optimal_g_constraint}
\end{align}
for a given MDP $(\mathcal{X},\mathcal{U}, n, P, c, c_n)$, penalty functions $(d, d_n)$, set of transition functions  $\mathcal{P}$, upper bound $l_0 \in \mathbb{R}$, and initial state $x_0 \in \mathcal{X}$.
\end{problem}


We impose the following assumptions on our formulation:

\begin{assumption} \label{assumption_1}
    The costs and penalties at each instance of time are upper bounded by the finite maximum values $c^{M} \in \mathbb{R}$ and $d^{M} \in \mathbb{R}$, respectively. They are also lower bounded by the finite minimum values $c^{m} \in \mathbb{R}$ and $d^{m} \in \mathbb{R}$, respectively.
\end{assumption}

Assumption \ref{assumption_1} ensures that the expected total cost \eqref{performance_g} and robust total penalty \eqref{exp_max_constraint} are finite for any value of $n \in \mathbb{N}$.

\begin{assumption} \label{assumption_2}
    The bound $l_0 \in \mathbb{R}$ is such that the set $\mathcal{G}_{l_0}:= \{\boldsymbol{g}\in \mathcal{G}\,|\, \mathcal{L}_0(\boldsymbol{g};x_0)\leq \l_0\}$ is not empty.
\end{assumption}

Assumption \ref{assumption_2} ensures that Problem \ref{problem_1} has a feasible solution and, thus, it is well-posed. Our goal is to efficiently compute an optimal solution to Problem \ref{problem_1} without violating the constraint. Next, we present a DP decomposition for the problem.

\section{Dynamic Programming Decomposition}
\label{section:feasible_sets_DP}
In this section,  we present the value functions that constitute a DP decomposition to compute the optimal control strategy $\boldsymbol{g}^*$ for Problem \ref{problem_1}.
To show that the computed strategy satisfies the distributionally robust constraint \eqref{optimal_g_constraint}, we need to prove its recursive feasibility for all $t=1, \ldots, n-1$. 
To achieve this, in Subsection \ref{Penalty_to_go_sets}, we define the \textit{penalty-to-go function} to express the application of the constraint only from any time $t$ to the terminal time $n$.
We then construct a set of upper bounds on the penalty-to-go function at any $t=0,\dots,n-1$, such that these bounds admit a feasible solution, and present a methodology to compute these sets. 
We also introduce the notion of bound functions, which will be utilized it to ensure recursive feasibility.
In Subsection \ref{DP_prob1}, we use bound functions within the proposed DP decomposition and prove its optimality.

\subsection{Feasible bound for robust constraint}\label{Penalty_to_go_sets}
\label{subsection:feasible_sets}

We begin by constructing the \textit{penalty-to-go} function that maps each realization of the state $x_t \in \mathcal{X}$ at any $t = 0,\dots,n-1$ to an expected worst-case penalty to reach $n$ using a sequence of control laws $g_{t:n-1} \in \prod_{\ell = t}^{n-1}\mathcal{G}_{\ell}$. Specifically, this penalty-to-go at each $t$ is
\begin{multline}
    \label{penalty_to_go}
    \mathcal{L}_t(g_{t:n-1};x_t) = \\ \max_{P_{t:n-1} \in \mathcal{P}^{n-t}} \mathbb{E}^{g_{t:n-1}}_{P_{t:n-1}}\Bigg[\sum_{\ell=t}^{n-1} d(X_\ell,U_\ell)
    + d_n(X_n)\,\Big|\,x_t\Bigg],
\end{multline}
where the expectation on all the random variables is with respect to the distributions generated by the choice of control laws $g_{t:n-1} \in \prod_{\ell = t}^{n-1}\mathcal{G}_\ell$. Importantly, the control laws utilized prior to time $t$ do not influence the penalty-to-go from time $t$.
Additionally, note that the penalty-to-go from $t=0$ is the total expected penalty in \eqref{exp_max_constraint}. Next, we construct a set of feasible upper bounds on the penalty-to-go function.

\begin{definition}\label{Def_lambda_set}
For all $t=1,\dots,n-1$, the \textit{set of feasible upper bounds} for a state $x_t \in \mathcal{X}$ is
\begin{gather}    
\label{feasible_bound}
    \hspace{-2pt}
    \Lambda_{t}(x_t) \hspace{-2pt}:= \hspace{-2pt}\Big\{ \hspace{-1pt}l_t \in \mathbb{R} \,|\, 
    \exists\; g_{t:n-1} \hspace{-3pt} \in \hspace{-3pt} \prod_{\ell = t}^{n-1}\mathcal{G}_\ell,
    \text{ s.t. } \mathcal{L}_{t}(g_{t:n-1},x_t) \hspace{-1pt}\leq \hspace{-1pt} l_t \hspace{-1pt} \Big\},
    \end{gather}    
    with $\Lambda_n(x_n) := [d_n(x_n), d^M]$ at $t=n$ for each $x_n \in \mathcal{X}$ and $\Lambda_0(x_0) := \{l_0\}$ identically for all $x_0 \in \mathcal{X}$. 
\end{definition} 
In Definition \ref{Def_lambda_set}, the bound $l_t$ acts only upon the penalty-to-go $\mathcal{L}_t(g_{t:n-1}; x_t)$. Thus, each bound $l_t \in \Lambda_t(x_t)$ ensures feasibility of only the control laws $g_{t:n-1} \in \prod_{\ell = t}^{n-1}\mathcal{G}_\ell$ for each $x_t \in \mathcal{X}$ and $t=0,\dots,n-1$. 

Next, to ensure recursive feasibility in our solution approach, our goal is to select a feasible bound on the penalty-to-go for all $t=0,\dots,n-1$. These bounds should ensure that, starting with $l_0$ at $t=0$, there exists at least one feasible sequence of control laws $g_{t:n-1} \in \prod_{\ell = t}^{n-1}\mathcal{G}_\ell$.
We note that based on Assumption \ref{assumption_2}, such a sequence exists at $t=0$.

To this end, we establish the notion of \textit{bound functions} $\lambda_t: \mathcal{X} \to \mathbb{R}$ at each $t=0,\dots,n$. 
The output of the bound function $\lambda_t(x_t)$ is a feasible bound from Definition \ref{Def_lambda_set} for all $x_t \in \mathcal{X}$ and all $t$.
Then, for any bound $l_t \in \Lambda_t(x_t)$ and a control action $u_t \in \mathcal{U}$, the set of \textit{recursively consistent bound functions} at time $t+1$ is 
\begin{align}
\label{F_t_set}
F_t(x_t,u_t,l_t) = \Big\{ \lambda_{t+1} \,\Big|\,
    \lambda_{t+1}(x_{t+1}) \in \Lambda_{t+1}(x_{t+1}), \nonumber\\
    \;\forall x_{t+1} \in \mathcal{X}\;\text{and} \nonumber \\     \max_{P_{t}\in \mathcal{P}} \mathbb{E}_{P_{t}}[\lambda_{t+1} (X_{t+1})\,|\,x_t,u_t]\leq l_t - d(x_t,u_t)\Big\}.       
\end{align}

The inequality in the conditioning of the set in \eqref{F_t_set} yields the allowable bound at time $t+1$ after considering the ``consumption" of the bound $l_t$ by the penalty $d(x_t, u_t)$ incurred at time $t$.
This inequality is imposed upon the maximum expected value of $\lambda_{t+1}(X_{t+1})$ given the state $x_t$ and action $u_t$ to ensure recursive constraint satisfaction.
Note that this maximization captures the distributionally robust form of transition functions in \eqref{penalty_to_go} and, thus, accounts for the possible influence of the attacker. 
Thus, given $l_t \in \Lambda_t(x_t)$ at time $t$, restricting attention to $\lambda_{t+1} \in F_t(x_t, u_t, l_t)$ ensures that any selected bound at time $t+1$ is feasible.
Beginning with $l_0$ at $t=0$ and applying this property for the set $F_t(x_0, u_0, l_0)$ for all $x_0 \in \mathcal{X}$ and $u_0 \in \mathcal{U}$ ensures recursive feasibility and constraint satisfaction for all $t=0,\dots,n$.
Due to the importance of the sets $\Lambda_t(x_t)$ in \eqref{F_t_set}, it is essential to efficiently compute them before deriving an optimal strategy. 

To begin, we observe that for any feasible $l_t \in \Lambda_t(x_t)$, there exists a sequence of control laws $g_{t:n-1}$ that satisfies the constraint $\mathcal{L}_{t}(g_{t:n-1}; x_t) \leq \hat{l}_t$ for all $l_t \leq \hat{l}_t \in \mathbb{R}$ and all $x_t \in \mathcal{X}$. Hence, it is sufficient to compute the smallest feasible bound $\lambda^m_t(x_t)$ for each $x_t \in \mathcal{X}$ and note that the set $\Lambda_t(x_t) \subseteq [\lambda^m_t(x_t), \infty)$. 
At the other extreme, without loss of generality, we can restrict the maxima of $\Lambda_t(x_t)$ to $l_t^M = \min\big\{l_0,\sum_{i = t}^{n} d_i^{M}\big\}$. This is because including bounds larger than $l_t^M$ does not increase the set of feasible sequences of control laws $g_{0:t-1}$.
Thus, the structural form of the set of feasible upper bounds is $\Lambda_t(x_t) = [\lambda_t^m(x_t), l_t^M]$ for all $t=0,\dots,n-1$.

Next, we present a recursive approach to compute $\lambda_t^m(x_t)$ for all $t$ and complete the construction of $\Lambda_t(x_t)$.
\begin{lemma} \label{Lemma_l_m}
The lower bound $\lambda^m_t(x_t)$ of the set $\Lambda_t(x_t)$ for all $x_t \in \mathcal{X}$ and $t= 1, \ldots, n-1$ is obtained by the following minimization problem 
\begin{multline}
    \label{dp_penalty}
    \lambda^m_t(x_t)= \min_{u_t \in \mathcal{U}}\Big\{ d(x_t,u_t) \\
    +\max_{P_{t}\in \mathcal{P}} \mathbb{E}_{P_{t}}\big[\lambda^m_{t+1}(X_{t+1})\,\big|\,x_t, u_t\big]\Big\}.
\end{multline}
\end{lemma}

\begin{proof}
The smallest feasible bound $\lambda^m_t(x_t)$ belongs to the set $\Lambda_t(x_t)$. From Definition \ref{Def_lambda_set}, we can see that there exists a sequence $g_{t:n-1}$ for which the penalty-to-go is exactly equal to $\lambda^m_t(x_t)$ and any bound smaller than $\lambda^m_t(x_t)$ is infeasible. Thus we can compute $\lambda^m_t(x_t)$ as 
\begin{equation}
\label{min_feasible_bound}
\lambda^m_t(x_t)=\min_{g_{t:n-1}\in \prod_{\ell = t}^{n-1}\mathcal{G}_\ell}\mathcal{L}_t(g_{t:n-1};x_t),
\end{equation}
which becomes an instance of the standard distributionally robust 
DP problem.
The objective is to minimize the penalty-to-go while being distributionally robust against the uncertainty in the transition function. Using the arguments presented in \cite[Theorem 2.1]{iyengar2005robust} for deriving the optimal objective in such a problem, we can see how Lemma \ref{Lemma_l_m} computes the minimum value of $\mathcal{L}_t(g_{t:n-1};x_t)$ at each $t$. This shows that, Lemma \ref{Lemma_l_m} can be used to recursively compute the smallest feasible bound $\lambda^m_t(x_t)$ for each $x_t \in \mathcal{X}$ and all $t$.
\end{proof}

\subsection{Dynamic program for Problem 1}\label{DP_prob1}
\label{subsection:DP}
In this subsection, before presenting the DP decomposition, we begin by defining the cost-to-go in a manner similar to the penalty-to-go in Subsection \ref{Penalty_to_go_sets}. 
For all $t=0,\dots,n-1$, the cost-to-go from any  $x_t \in \mathcal{X}$ is
\begin{gather}
    \label{cost_to_go}    \mathcal{J}_t(g_{t:n-1};x_t)=\mathbb{E}^{g_{t:n-1}}
    \Bigg[\sum_{\ell=t}^{n-1} c(X_\ell,U_\ell)
    + c_n(X_n)\,\Big|\,x_t\Bigg],   
\end{gather}
where $\mathbb{E}^{g_{t:n-1}}$ denotes the expectation on all the random variables with respect to the distributions generated by the nominal transition function $\bar{P}$ and the choice of control laws $g_{t:n-1} \in \prod_{\ell=t}^{n-1} \mathcal{G}_{\ell}$.
Note that the cost-to-go at time $t$ is affected only by the sequence of control laws $g_{t:n-1}$ and the cost-to-go at $t=0$ is equivalent to the performance measure \eqref{performance_g} for any strategy $\boldsymbol{g}$. 

Before we can construct a DP decomposition, we recall that Problem \ref{problem_1} also restricts the set of feasible strategies by constraining the penalty-to-go $\mathcal{L}_0(\boldsymbol{g}; x_0)$ with an upper bound $l_0$. 
Thus, as derived in Subsection \ref{Penalty_to_go_sets}, we need to impose a constraint using the bound function $\lambda_t \in F_{t-1}(x_{t-1}, u_{t-1}, l_{t-1})$ for all $t=1,\dots,n$ to ensure recursive feasibility in our solution approach. 
To this end, at each $t=0, \ldots, n-1$, we expand our state-space $\mathcal{X}$ by appending with it a set of possible bounds, $\mathbb{R}$.
Thus, the value functions of our DP decomposition are functions of $(X_t,L_t) \in \mathcal{X} \times \mathbb{R}$, where the random variable $L_t = \lambda_t(X_t)$.
The realizations of the random variable $L_t$ are denoted by $l_t$.
Furthermore, at each $x_t \in  \mathcal{X}$, the control law $g_{t} \in \mathcal{G}_t$ at each $t=0, \ldots, n-1$ selects a control action $U_t \in \mathcal{U}$ using the expanded state space as $U_t=g_t(X_t,L_t)$.

\begin{remark}
    We note that expanding the state space from $X_t$ to $(X_t, L_t)$ expands the domain of control laws as compared to the standard Markovian control law for regular MDPs. However, the result in Subsection \ref{Penalty_to_go_sets} is still valid for control laws with this larger domain because the functions introduced in \ref{Penalty_to_go_sets} depend only on the realization $x_t \in \mathcal{X}$ of $X_t$ and are independent of the realization $l_t = \lambda_t(x_t)$ of $L_t$.
\end{remark}

For all $t=0, \ldots, n-1$, the value function for all $x_t \in \mathcal{X}$ and $l_t \in \Lambda(x_t)$ corresponding to the sequence of control laws $g_{t:n-1}$ is given by
\begin{equation}
\label{Val_func_strategy}
V_t^{g_{t:n-1}}(x_t,l_t)=
\begin{cases}
    \mathcal{J}_t(g_{t:n-1};x_t) & \text{if}\; 
    \mathcal{L}_t(g_{t:n-1};x_t)\leq l_t,  \\
    \kappa & \text{otherwise},   
\end{cases}
\end{equation}
where $\kappa \in \mathbb{R}$ is a large  constant that satisfies $\kappa > n \cdot c^M$ and indicates constraint violation by $g_{t:n-1}$. 
Eventually, when we minimize over the set of strategies, the presence $\kappa$ will help us exclude infeasible solutions.
At the terminal time $n$, where no actions are allowed, the value function is simply $V_n(x_n,l_n)=c(x_n)$.
Then, the optimal value functions for all $x_t \in \mathcal{X}$, $l_t = \lambda_t(x_t)$ and all $t=0,\dots, n-1$ are
\begin{equation}
\label{optimal_Val_func}
    V_t(x_t,l_t) = \min_{g_{t:n-1} \in \prod_{\ell=t}^{n-1} \mathcal{G}_{\ell}} V_t^{g_{t:n-1}}(x_t,l_t).
\end{equation}

\begin{theorem}\label{theorem_dp} 
At each $t=0, \ldots, n-1$, for all $x_t \in \mathcal{X}$ and $l_t = \lambda_t(x_t)$, the optimal value function can be recursively computed using the following DP decomposition:
\begin{align}
\label{dp}
V_{t}(x_t,l_t)=\min_{\substack{ u_t\in \mathcal{U},\\   
    \lambda_{t+1}\in F_t(x_t,u_t,l_t)}}
    \Bigg\{ c(x_t,u_t) + \nonumber\\
    \mathbb{E}\Big[V_{t+1}(X_{t+1},\lambda_{t+1}(X_{t+1})) \,|\, x_t,u_t\Big]\Bigg\},   
\end{align}
where, at the terminal time $t=n$, the optimal value function is simply given by $V_n(x_n,l_n)=c(x_n)$. 
\end{theorem}

\begin{proof}
We prove that the DP decomposition presented in Theorem \ref{theorem_dp} computes the optimal value function recursively using mathematical induction.
At the terminal time, the value function is given by $V_n(x_n,l_n)= c(x_n)$.
Suppose that the optimal value function $V_{t+1}$ at time $t+1$ can be computed according to \eqref{dp}. It is enough to show that \eqref{dp} can be used to compute $V_t(x_t,l_t)$ at time $t$.
We need first to show that the left-hand side in \eqref{dp} is lower bounded by the right-hand side and vice-versa. 
As a result, the left-hand side of \eqref{dp} will be both upper and lower bounded by the expression on the right-hand side.
Hence, we conclude that in \eqref{dp}, the left-hand side is equal to the right-hand side. Details of the mathematical arguments are provided in Appendix A. 
\end{proof}

\begin{remark} 
We showed that the DP decomposition presented in \eqref{dp} computes the optimal value function at each $t=0, \ldots, n-1$. Using Theorem \ref{theorem_dp}, we can compute the sequence of optimal control laws $g^*_{0:n-1}\in \prod_{\ell=0}^{n-1} \mathcal{G}_{\ell}$ which yields the optimal value function $V_0(x_0,l_0)$ at time $t=0$.
\end{remark}

\begin{remark} 
At any $t=0, \ldots, n-1$, and for all $x_t \in \mathcal{X}$ and a feasible bound $l_t$, Theorem \ref{theorem_dp} states that the optimal control action is $u^*_t=g^*_t(x_t,l_t)$, i.e., the minimizing argument in \eqref{dp}.
Subsequently, the optimal bound function  $\lambda^*_{t+1}(\cdot)$ is computed as a function of the state $x_t$, bound $l_t$, and optimal action $u^*_t$.
Since the optimal bound function is computed at the preceding time step, during implementation, $\lambda_t^*$ is available at the onset of time $t$ and the agent ensures that $l_t = \lambda_t^*(x_t)$.
Hence, to solve Problem \ref{problem_1}, the control action at all $t$ is selected as $u^*_t=g^*_t(x_t,\lambda^*_{t}(x_t))$. This shows that the optimal control strategy can be selected using $x_t \in \mathcal{X}$ during implementation.
\end{remark}

\section{Numerical Example}
\label{section:numerical_example}
In this Section, we illustrate the efficiency of the effectiveness of our approach using a numerical example.
We consider a reach-avoid problem where an agent seeks to navigate to a designated cell in a $4 \times 4$ grid world while avoiding a different cell in the grid. 
At each $t=0,\dots,n$, the agent's position $X_{t}$ takes values in the set of grid cells $\mathcal{X} = \big\{(0,0),(0,1),\dots,(3,2),(3,3)\big\}$. 
The action $U_t$ denotes the agent's direction of movement and takes values in the set $\mathcal{U} = \{(-1,0),(1,0),(0,0),(0,1),(0,-1)\}$.  
Under normal system operation, the agent has a small chance of movement failure by ``slipping." This is modeled by considering that at each $t$, the agent moves in the direction selected by the action $U_t$ with probability $0.8$ and may slip by moving in either the clockwise or anticlockwise direction to $U_t$ with probabilities of $0.1$ each.
The agent does not slip when selecting the action $(0,0)$, i.e., when deciding not to move.
Thus, starting at a randomly selected initial state $x_0 \in \mathcal{X}$, the agent's dynamics for all for all $t=0,\dots,n-1$ are
\begin{align}
\label{num_sim_dynamics}
    P(x_{t+1}~|~x_t, u_t) = 
    \begin{cases}
        0.8 \quad \text{if } x_{t+1} = x_t + u_t, \\
        0.1 \quad \text{if } x_{t+1} = x_t + u_t^{\text{cl}}, \\
        0.1 \quad \text{if } x_{t+1} = x_t - u_t^{\text{cl}},
    \end{cases}
\end{align}
where if $u_t = (u_t^1, u_t^2)$, then $u_t^{\text{cl}} = (-u_t^2, u_t^1)$ is the clockwise rotation of $u_t$.
If the agent's position $X_t$ is at one of the four corners of the grid or along the edge of the grid, then the agent may only slip only in the available directions and not move off the grid. Thus, if only one direction is available, they move in the direction of the selected action with a probability of $0.8$ and move in the available direction with a probability of $0.2$.



The goal of the agent is to reach the destination cell $(3,2)$ while avoiding a ``trap" cell $(2,1)$.
Thus, after a time horizon of $n=10$ time steps, the agent incurs a terminal cost given by the distance from the position $X_{10}$ and the target $(3,2)$, i.e., $c_n(X_n) = \eta\big(X_n, (3,2)\big)$ where $\eta(\cdot, \cdot)$ denotes the Manhattan distance. The agent incurs no interim costs, i.e., $c(X_t, U_t) := 0$ for all $t=0,\dots,n-1$.
Furthermore, the agent incurs a penalty of $1$ unit at any instance of time if their position coincides with the trap $(2,1)$, i.e., $d(X_t, U_t) = \mathbb{I}\big[X_t = (3,2)\big]$ for all $t=0,\dots,n$.
An adversary, if present, attacks the reliability of the agent's actuator. Thus, under an attack, the probability of slipping may increase. We incorporate vulnerability to attacks by defining, on the tuple of actual movements $(u_t, u_t^{\text{cl}}, -u_t^{\text{cl}})$ for a given action $u_t \in \mathcal{U}$, the set of possible probability distributions $\mathcal{P} := \big\{(0.7, 0.3, 0),(0.7, 0.2, 0.1), \dots, (0.7, 0, 0.3), (0.8, 0.2, 0),$ $(0.8, 0.1, 0.1),(0.8, 0, 0.2)\big\}$. 

We run $5000$ simulations for two initial positions, $(1,0)$ and $(0,1)$ with $l_0=2.5$, for which the heat map of the path selected by the agent is given in Fig. \ref{fig:First_initial_condition} and Fig. \ref{fig:Second_initial_condition} respectively.
In each simulation, the system is vulnerable to attack
, and the transition probability is randomly picked from the set $\mathcal{P}$ as given above to emulate the attack.
For demonstration, the computed optimal control strategy is implemented in a receding horizon manner for $200$ time steps. 
To validate how our approach provides more control over the trade-off between conservativeness and optimality, we compare three cases for each initial condition.
In the first case, we consider the set of probability distributions $\mathcal{P}$ as given above to compute the distributionally robust control strategy. In this case, the agent visits the trap cell $497$ and $250$ times as shown in Fig. \ref{fig:1a} and Fig. \ref{fig:2a}, respectively.
For the second case, we expand the set $\mathcal{P}$ to include every possible distribution in $\Delta(\mathcal{X})$ to compute a conservative strategy. As a result, in Fig. \ref{fig:1b} and Fig. \ref{fig:2b}, the number of times the agent moves into the trap cell are $35$ and $10$, respectively.
Lastly, we compute a stochastic strategy by considering only the nominal transition distribution to be the set $\mathcal{P}$. Accordingly, in Fig. \ref{fig:1c} and Fig. \ref{fig:2c}, the agent moves $764$ and $363$ times into the trap cell, respectively.
We observe that when the agent utilizes the distributionally robust control strategy. However, it visits the trap cell more often, it quickly reaches the target cell when compared to the conservative strategy. On the other hand, it reaches the target cell as quickly as the stochastic strategy, with fewer visits to the trap cell.

\begin{figure*}[h!]
    \centering
    \begin{subfigure}[t]{0.31\textwidth}
        \includegraphics[width=\textwidth, keepaspectratio]        {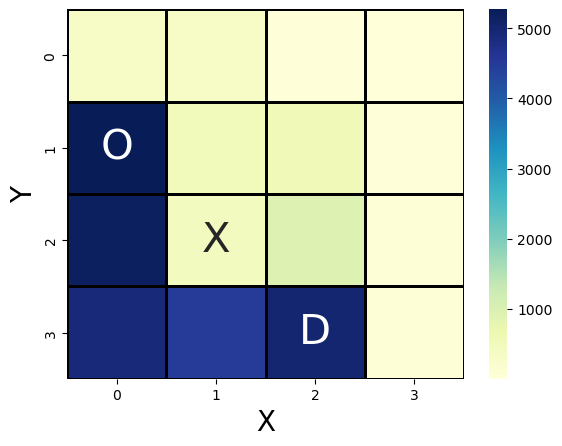}
        \caption{}
        \label{fig:1a}
    \end{subfigure}
    \hspace{5pt}
    \vline
    \hspace{0pt}
    \hfill
    \begin{subfigure}[t]{0.31\textwidth}
        \includegraphics[width=\textwidth, keepaspectratio]{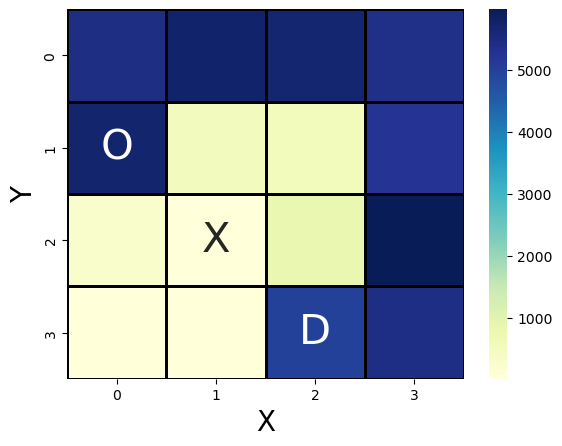}
        \caption{}
        \label{fig:1b}
    \end{subfigure}
    \hspace{5pt}
    \vline
    \hspace{0pt}
    \hfill
    \begin{subfigure}[t]{0.31\textwidth}
        \includegraphics[width=\textwidth, keepaspectratio]
        {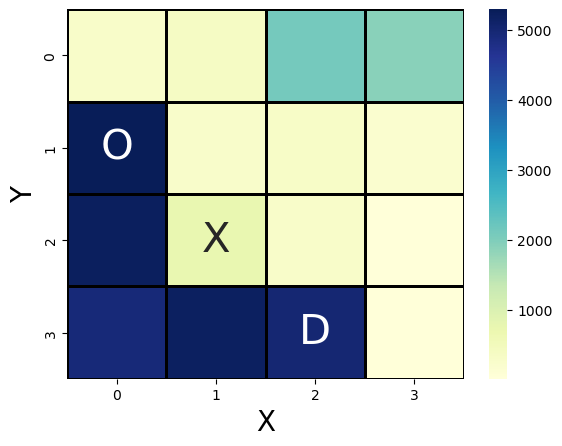}
        \caption{}
        \label{fig:1c}
    \end{subfigure}
    \caption{For the initial state $(1,0)$, the strategy implemented in : (a) distributionally robust (b) conservative (c) stochastic}
    \label{fig:First_initial_condition}
\end{figure*}

\vspace{15pt}

\begin{figure*}[h!]
    \centering
    \begin{subfigure}[t]{0.31\textwidth}
        \includegraphics[width=\textwidth, keepaspectratio]                {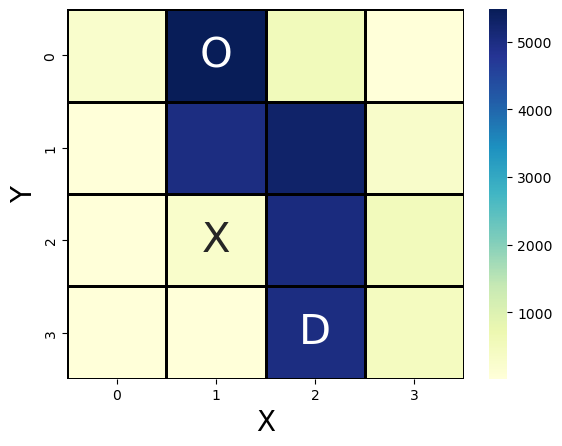}
        \caption{}
        \label{fig:2a}
    \end{subfigure}
    \hspace{5pt}
    \vline 
    \hspace{0pt}
    \hfill
    \begin{subfigure}[t]{0.31\textwidth}
        \includegraphics[width=\textwidth, keepaspectratio]
        {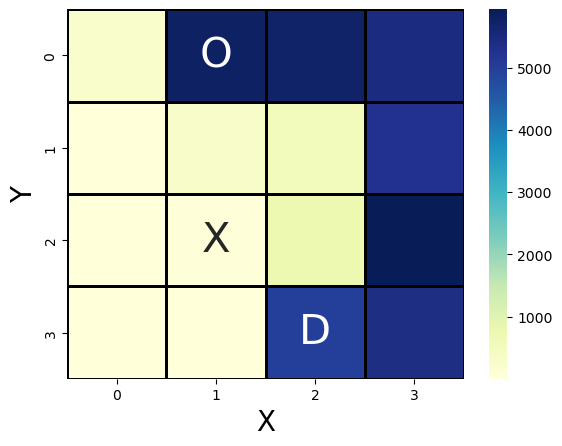}
        \caption{}
        \label{fig:2b}
    \end{subfigure}
    \hspace{5pt}
    \vline
    \hfill
    \begin{subfigure}[t]{0.31\textwidth}
        \includegraphics[width=\textwidth, keepaspectratio]
        {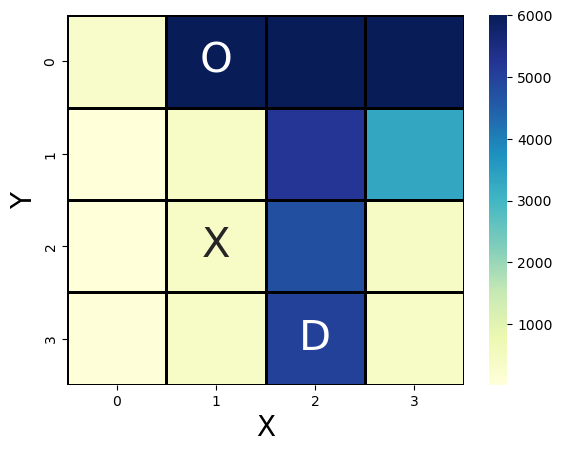}

        \caption{}
        \label{fig:2c}
    \end{subfigure}
    \caption{For the initial state $(0,1)$, the strategy implemented in : (a) distributionally robust (b) conservative (c) stochastic}
    \label{fig:Second_initial_condition}
    \vspace{-5pt}
\end{figure*}

\section{Conclusion}
\label{section:conclusion}

In this paper, we proposed the problem of controlling a CPS, which is vulnerable to attack as a distributionally robust stochastic cost minimization problem.
For this problem, we presented DP decomposition to compute the optimal control strategy, which ensures the recursive feasibility of the distributionally robust constraint. Finally, we illustrated the utility of our solution approach using a numerical example. Future work should consider using these results in tandem with fast computation techniques for applications with large state space like human-robot collaboration tasks, power grids, and connected and automated vehicles. In such applications, it is essential to avoid over-conservatism while maintaining resilience against any vulnerabilities.

\bibliographystyle{ieeetr}
\bibliography{References, Latest_IDS}

\begin{thebibliography}{10}

\bibitem{kim2012cyber}
K.-D. Kim and P.~R. Kumar, ``Cyber--physical systems: A perspective at the
  centennial,'' {\em Proceedings of the IEEE}, vol.~100, no.~Special Centennial
  Issue, pp.~1287--1308, 2012.

\bibitem{Malikopoulos2021}
A.~A. Malikopoulos, ``On team decision problems with nonclassical information
  structures,'' {\em IEEE Transactions on Automatic Control}, vol.~68, no.~7,
  pp.~3915--3930, 2023.

\bibitem{Malikopoulos2020}
A.~A. Malikopoulos, L.~E. Beaver, and I.~V. Chremos, ``Optimal time trajectory
  and coordination for connected and automated vehicles,'' {\em Automatica},
  vol.~125, no.~109469, 2021.

\bibitem{Nishanth2023AISmerging}
N.~Venkatesh, V.-A. Le, A.~Dave, and A.~A. Malikopoulos, ``Connected and
  automated vehicles in mixed-traffic: Learning human driver behavior for
  effective on-ramp merging,'' in {\em Proceedings of the 62nd IEEE Conference
  on Decision and Control (CDC)}, 2023 (to appear, arXiv:2304.00397).

\bibitem{ansere2020optimal}
J.~A. Ansere, G.~Han, L.~Liu, Y.~Peng, and M.~Kamal, ``Optimal resource
  allocation in energy-efficient internet-of-things networks with imperfect
  csi,'' {\em IEEE Internet of Things Journal}, vol.~7, no.~6, pp.~5401--5411,
  2020.

\bibitem{Dave2020SocialMedia}
A.~Dave, I.~V. Chremos, and A.~A. Malikopoulos, ``Social media and misleading
  information in a democracy: A mechanism design approach,'' {\em IEEE
  Transactions on Automatic Control}, vol.~67, no.~5, pp.~2633--2639, 2022.

\bibitem{baezner2017stuxnet}
M.~Baezner and P.~Robin, ``Stuxnet,'' tech. rep., ETH Zurich, 2017.

\bibitem{serror2020challenges}
M.~Serror, S.~Hack, M.~Henze, M.~Schuba, and K.~Wehrle, ``Challenges and
  opportunities in securing the industrial internet of things,'' {\em IEEE
  Transactions on Industrial Informatics}, vol.~17, no.~5, pp.~2985--2996,
  2020.

\bibitem{ghiasi2023comprehensive}
M.~Ghiasi, T.~Niknam, Z.~Wang, M.~Mehrandezh, M.~Dehghani, and N.~Ghadimi, ``A
  comprehensive review of cyber-attacks and defense mechanisms for improving
  security in smart grid energy systems: Past, present and future,'' {\em
  Electric Power Systems Research}, vol.~215, p.~108975, 2023.

\bibitem{10155711}
L.~Zhang, K.~Sridhar, M.~Liu, P.~Lu, X.~Chen, F.~Kong, O.~Sokolsky, and I.~Lee,
  ``Real-time data-predictive attack-recovery for complex cyber-physical
  systems,'' in {\em 2023 IEEE 29th Real-Time and Embedded Technology and
  Applications Symposium (RTAS)}, pp.~209--222, 2023.

\bibitem{Dave2021nestedaccess}
A.~Dave, N.~Venkatesh, and A.~A. Malikopoulos, ``Decentralized control of two
  agents with nested accessible information,'' in {\em 2022 American Control
  Conference (ACC)}, pp.~3423--3430, IEEE, 2022.

\bibitem{sutton2018reinforcement}
R.~S. Sutton and A.~G. Barto, {\em Reinforcement learning: An introduction}.
\newblock MIT press, 2018.

\bibitem{varagapriya2023chance}
V.~Varagapriya and V.~V. Singh, ``Chance-constrained formulation of mdps under
  total reward criteria: an application to advertisement model,'' in {\em 2023
  European Control Conference (ECC)}, pp.~1--6, IEEE, 2023.

\bibitem{altman2021constrained}
E.~Altman, {\em Constrained Markov decision processes}.
\newblock Routledge, 2021.

\bibitem{ermon2012probabilistic}
S.~Ermon, C.~Gomes, B.~Selman, and A.~Vladimirsky, ``Probabilistic planning
  with non-linear utility functions and worst-case guarantees,'' in {\em
  Proceedings of the 11th International Conference on Autonomous Agents and
  Multiagent Systems-Volume 2}, pp.~965--972, 2012.

\bibitem{Malikopoulos2022a}
A.~A. Malikopoulos, ``Separation of learning and control for cyber-physical
  systems,'' {\em Automatica}, vol.~151, no.~110912, 2023.

\bibitem{mannor2007bias}
S.~Mannor, D.~Simester, P.~Sun, and J.~N. Tsitsiklis, ``Bias and variance
  approximation in value function estimates,'' {\em Management Science},
  vol.~53, no.~2, pp.~308--322, 2007.

\bibitem{bertsekas1973sufficiently}
D.~Bertsekas and I.~Rhodes, ``Sufficiently informative functions and the
  minimax feedback control of uncertain dynamic systems,'' {\em IEEE
  Transactions on Automatic Control}, vol.~18, no.~2, pp.~117--124, 1973.

\bibitem{iyengar2005robust}
G.~N. Iyengar, ``Robust dynamic programming,'' {\em Mathematics of Operations
  Research}, vol.~30, no.~2, pp.~257--280, 2005.

\bibitem{gagrani2017decentralized}
M.~Gagrani and A.~Nayyar, ``Decentralized minimax control problems with partial
  history sharing,'' in {\em 2017 American Control Conference (ACC)},
  pp.~3373--3379, IEEE, 2017.

\bibitem{shoukry2013minimax}
Y.~Shoukry, J.~Araujo, P.~Tabuada, M.~Srivastava, and K.~H. Johansson,
  ``Minimax control for cyber-physical systems under network packet scheduling
  attacks,'' in {\em Proceedings of the 2nd ACM international conference on
  High confidence networked systems}, pp.~93--100, 2013.

\bibitem{Dave2021minimax}
A.~Dave, N.~Venkatesh, and A.~A. Malikopoulos, ``On decentralized minimax
  control with nested subsystems,'' in {\em 2022 American Control Conference
  (ACC)}, pp.~3437--3444, IEEE, 2022.

\bibitem{Dave2023approximate}
A.~Dave, N.~Venkatesh, and A.~A. Malikopoulos, ``{Approximate Information
  States for Worst-Case Control and Learning in Uncertain Systems},'' {\em
  arXiv:2301.05089 (in review)}, 2023.

\bibitem{rasouli2018scalable}
M.~Rasouli, E.~Miehling, and D.~Teneketzis, ``A scalable decomposition method
  for the dynamic defense of cyber networks,'' in {\em Game Theory for Security
  and Risk Management}, pp.~75--98, Springer, 2018.

\bibitem{zhu2011robust}
Q.~Zhu and T.~Ba\c{s}ar, ``Robust and resilient control design for
  cyber-physical systems with an application to power systems,'' in {\em 2011
  50th IEEE Conference on Decision and Control and European Control
  Conference}, pp.~4066--4071, IEEE, 2011.

\bibitem{coraluppi2000mixed}
S.~P. Coraluppi and S.~I. Marcus, ``Mixed risk-neutral/minimax control of
  discrete-time, finite-state markov decision processes,'' {\em IEEE
  Transactions on Automatic Control}, vol.~45, no.~3, pp.~528--532, 2000.

\bibitem{chen2004dynamic}
R.~C. Chen and G.~L. Blankenship, ``Dynamic programming equations for
  discounted constrained stochastic control,'' {\em IEEE transactions on
  automatic control}, vol.~49, no.~5, pp.~699--709, 2004.

\end{thebibliography}

\section*{Appendix A - Proof of Theorem 1}
We present a detailed proof of Theorem \ref{theorem_dp} by elaborating on how we show that left-hand side of \eqref{dp} is both upper and lower bounded by the expression on the right-hand side.

To begin, we recall the induction assumption that the optimal value function $V_{t+1}$ at time $t+1$ can be computed according to \eqref{dp}.

At time $t$, for each $x_t \in \mathcal{X}$ and a feasible bound $l_t = \lambda_t(x_t)$, let $g^*_{t:n-1} \in \prod_{\ell=t}^{n-1} \mathcal{G}_{\ell}$ be the sequence of control laws such that 
\begin{equation}
    g^*_{t:n-1}= \argmin_{g_{t:n-1} \in \prod_{\ell=t}^{n-1} \mathcal{G}_{\ell}} V_t^{g_{t:n-1}}(x_t,l_t).
\end{equation}
The value function corresponding to this sequence of control laws is the optimal value function at time $t$ as given by
\begin{equation}
    V_t(x_t,l_t)=V_t^{g^*_{t:n-1}}(x_t,l_t).
\end{equation}
We expand the expression for the value function of this sequence of control laws as
\begin{multline}
V_t(x_t,l_t)=\mathbb{E}^{g^*_{t:n-1}}\Bigg[\sum_{\ell=t}^{n-1} c(X_\ell,U_\ell) + c_n(X_n)\,\Big|\,x_t\Bigg],\\
=c(x_t,g^*_t(x_t,l_t))+\mathbb{E}^{g^*_{t:n-1}}\Bigg[\sum_{\ell=t+1}^{n-1} c(X_\ell,U_\ell) + c_n(X_n)\,\Big|\,x_t\Bigg],\label{dphard_proof_2}
\end{multline} 
where, we note that the penalty-to-go of the sequence $g^*_{t:n-1}$ upper bounded by $l_t$. Hence, we only analyze the cost-to-go component of the value function for this sequence of control laws. We use the law of iterated expectations to introduce $X_{t+1}$ into the inner expectation as
\begin{multline}
\label{dphard_proof_3}
V_t(x_t,l_t)=c(x_t,g_t^*(x_t,l_t)) + \\\mathbb{E}\Bigg[ \mathbb{E}^{g_{t+1:n-1}^*}\Bigg[\sum_{\ell=t+1}^{n-1} c(X_\ell,U_\ell)+ c_n(X_n)\,\Big|\,X_{t+1},\Bigg]\\
\,\Big|\,x_t,g_t^*(x_t,l_t)\Bigg].
\end{multline}
Next, using \eqref{cost_to_go}, we write the inner expectation as the cost-to-go from time $t+1$ for the sequence ${g}^*_{t+1:n-1}$: 
\begin{multline}
    \label{dphard_proof_4} V_t(x_t,l_t)=c(x_t,g_t^*(x_t,l_t)) + \\\mathbb{E}\Bigg[ \mathcal{J}_{t+1}(g_{t+1:n-1}^*;X_{t+1})\,\Big|\,x_t,g_t^*(x_t,l_t)\Bigg ].    
\end{multline}
From \eqref{Val_func_strategy} and \eqref{optimal_Val_func}, we can see that the cost-to-go for ${g}^*_{t+1:n-1}$ is always as big as the optimal value function at $t+1$. Considering this, in addition to the induction hypothesis, we can write:
\begin{multline}
    \label{dphard_proof_5}    
    V_t(x_t,l_t)\geq c(x_t,g_t^*(x_t,l_t)) + \\
    \mathbb{E}\Bigg[ V_{t+1}(X_{t+1},\mathcal{L}_{t+1}(g_{t+1:n-1}^*;X_{t+1}))\,\Big|\,x_t,g_t^*(x_t,l_t)\Bigg],
\end{multline}
where, we note that $\mathcal{L}_{t+1}(g_{t+1:n-1}^*,X_{t+1})$ is exactly the penalty-to-go for the sequence of control laws $g^*_{t+1:n-1}$, hence it is a feasible bound at time $t+1$.
By minimizing the right-hand side  of  \eqref{dphard_proof_5} over feasible combination of the action $u_t$ and future bound function $\lambda_{t+1}$ we have: 
\begin{multline}
    \label{dphard_proof_6} 
    V_t(x_t,l_t)\geq \min_{\substack{ u_t\in \mathcal{U},\\   
    \lambda_{t+1}\in F_t(x_t,u_t,l_t)}}\Bigg\{ c(x_t,u_t) + \\
\mathbb{E}\Big[V_{t+1}(X_{t+1},\lambda_{t+1}(X_{t+1})) \,|\, x_t,u_t\Big]\Bigg\}.
\end{multline}
This shows that the left-hand side of \eqref{dp}, is greater than the right-hand side. 

For each time $t$ and $x_t \in \mathcal{X}$ with the bound imposed being $l_t=\lambda_t(x_t)$, we consider   
\begin{multline}
\label{dphard_proof_7}
(u^*_t,\lambda^*_{t+1}) = \argmin_{\substack{ u_t\in \mathcal{U},\\   
    \lambda_{t+1}\in F_t(x_t,u_t,l_t)}} \Bigg\{ c(x_t,u_t) +\\
\mathbb{E}\Big[V_{t+1}(X_{t+1},\lambda_{t+1}(X_{t+1})) \,|\, x_t,u_t\Big]\Bigg\},
\end{multline}
where, $u^*_t$ is the optimal action at time $t$ and $\lambda^*_{t+1}$ is the corresponding optimal bound function at time $t+1$.
Now, we construct a sequence of control laws $\hat{g}_{t:n-1}$ such that:
\begin{align}
\label{g_hat_cons1}
\hat{g}_t(x_t,l_t)&=u^*_t\\  
\hat{g}_{t+\ell}&=g^*_{t+\ell},\text{ for } \ell= 1, \ldots, (n-1-t). 
\label{g_hat_cons11}
\end{align}
This implies that at time $t+1$, the newly constructed sequence of control laws $\hat{g}_{t+1:n-1}$ satisfies the following constraint
\begin{equation}
\label{g_hat_cons2}
\mathcal{L}_t(\hat{g}_{t+1:n-1};x_{t+1})\leq \lambda^*_{t+1}(x_{t+1}) \;\forall\; x_{t+1} \in  \mathcal{X}. 
\end{equation}
Next, by using \eqref{cost_to_go}, we establish the cost-to-go of the sequence of control laws $\hat{g}_{t:n-1}$ as an upper bound to the optimal value function at $t$. This is given by 
\begin{multline}\label{dphard_proof_9}
    V_t(x_t,l_t)\leq \mathcal{J}_t(\hat{g}_{t:n-1};x_t)\\    =\mathbb{E}^{\hat{g}_{t:n-1}}\Bigg[\sum_{\ell=t}^{n-1} c(X_\ell,U_\ell) + c_n(X_n) \,\Big|\,x_t\Bigg].    
\end{multline}
Then, we use the law of iterated expectations on the right-hand side of \eqref{dphard_proof_9} to introduce the random variable $X_{t+1}$, which yields
\begin{multline}\label{dphard_proof_10}
V_t(x_t,l_t)\leq c(x_t,\hat{g}_{t}(x_t,l_t)) + \\\mathbb{E}\Bigg[ \mathbb{E}^{\hat{g}_{t+1:n-1}}\Bigg[\sum_{\ell=t+1}^{n-1} c(X_\ell,U_\ell)+ c_n(X_n)\\\,\Big|\,X_{t+1},\Bigg]\,\Big|\,x_t,\hat{g}_{t}(x_t,l_t)\Bigg ],
\end{multline}
where we note that the inner expectation only depends on the choice of the sequence $\hat{g}_{t+1:n-1}$.
We use \eqref{Val_func_strategy} and \eqref{g_hat_cons2} to write the inner expectation as the value function of this sequence of control laws, given by
\begin{multline}\label{dphard_proof_11}
V_t(x_t,l_t)\leq c(x_t,\hat{g}_{t}(x_t,l_t)) + \\\mathbb{E}\Bigg[ V_{t+1}^{\hat{g}_{t+1:n-1}}(X_{t+1}, \lambda^*_{t+1}(X_{t+1}) )\,\Big|\,x_t, \hat{g}_{t}(x_t,l_t)\Bigg ].
\end{multline}
By the construction given in \eqref{g_hat_cons11}, we re-write the value function for the sequence of control laws $\hat{g}_{t+1:n-1}$ as 
\begin{multline}\label{dphard_proof12}
V_t(x_t,l_t)\leq c(x_t,\hat{g}_{t}(x_t,l_t)) +\\\mathbb{E}\Bigg[ V_{t+1}(X_{t+1},\lambda^*_{t+1}) \,\Big|\,x_t,\hat{g}_{t}(x_t,l_t)\Bigg ],\end{multline}
where, using \eqref{g_hat_cons1}, we recall that at time $t$ the control law $\hat{g}_t$ picks the optimal action $u^*_t$. Hence, 
\begin{multline}\label{dphard_proof12}
V_t(x_t,l_t)\leq c(x_t,u^*_t) +\\\mathbb{E}\Bigg[ V_{t+1}(X_{t+1},\lambda^*_{t+1}) \,\Big|\,x_t, \hat{g}_{t}(x_t,l_t)\Bigg ].\end{multline}

Now, we can clearly see that the optimal value function at time $t$ satisfies:
\begin{multline}
    \label{dphard_proof13} 
    V_t(x_t,l_t)\leq \min_{\substack{ u_t\in \mathcal{U},\\   
    \lambda_{t+1}\in F_t(x_t,u_t,l_t)}} \Bigg\{ c(x_t,u_t) + \\
\mathbb{E}\Big[V_{t+1}(X_{t+1},\lambda_{t+1}(X_{t+1})) \,|\, x_t,u_t\Big]\Bigg\}.
\end{multline}

The inequalities in \eqref{dphard_proof_6} and \eqref{dphard_proof13} prove that the optimal value function at time $t$ is both upper and lower bounded by the right-hand side of \eqref{dp}. 

\end{document}